\documentclass[a4paper, 11pt]{article}
\usepackage{amsmath, amsthm, amssymb, enumerate, graphicx}

\usepackage{geometry}
\usepackage{color}

\newcommand{\STAB}{\mathrm{STAB}}
\newcommand{\ac}{$\alpha$-critical}
\newcommand{\fdg}{facet-defining graph} 
\newcommand{\cfg}{critical facet-graph} 
 
\newcommand{\mws}{maximum worth set}
\newcommand{\os}{odd subdivision}
\newcommand{\uos}{unit odd subdivision} 
\newcommand{\invos}{shrinking}
\newcommand{\talpha}{\beta}
\newcommand{\adm}{admissible}
\newcommand{\at}{acyclic tournament}
\newcommand{\R}{\mathbb{R}}
\newcommand{\Z}{\mathbb{Z}}
\newcommand{\one}{\mathrm{1\hspace{-.5ex}l}}
\newcommand{\Plo}{P_\mathrm{LO}}

\newtheorem{prop}{Proposition}
\newtheorem{cor}{Corollary}
\newtheorem{theorem}{Theorem}
\newtheorem{lemma}{Lemma}
\newtheorem{claim}{Claim}

\title{Weighted graphs defining facets: a connection between\\ 
stable set and linear ordering polytopes%
\thanks{This work was partially supported by the {\em Actions
de Recherche Concert\'ees (ARC)\,} fund of the {\em Communaut\'e
fran\c{c}aise de Belgique}.}
}

\author{
Jean-Paul Doignon\thanks{Universit\'e Libre de Bruxelles, D\'epartement de Math\'ematique,
c.p.~216,  B-1050 Bruxelles, Belgium,  doignon@ulb.ac.be.} 
\and Samuel Fiorini\thanks{Universit\'e Libre de Bruxelles, D\'epartement de Math\'ematique,
c.p.~216,  B-1050 Bruxelles, Belgium,  sfiorini@ulb.ac.be.}
\and Gwena\"el Joret\thanks{Universit\'e Libre de Bruxelles, D\'epartement d'Informatique,
c.p.~212,  B-1050 Bruxelles, Belgium, gjoret@ulb.ac.be. G. Joret is a Research Fellow of the Fonds 
National de la Recherche Scientifique (F.R.S.--FNRS).}
}

\date{}

\begin{document}

\maketitle

\begin{abstract}
A graph is $\alpha$-critical if its stability
number increases whenever an edge is removed from its edge set.  The
class of $\alpha$-critical graphs has several nice structural
properties, most of them related to their defect
which is the number of vertices minus two times the stability number.  
In particular, a
remarkable result of Lov\'asz (1978) is the finite basis theorem for
$\alpha$-critical graphs of a fixed defect.  
The class of
$\alpha$-critical graphs is also of interest for at least two topics
of polyhedral studies.  First, Chv\'atal (1975) shows that each
$\alpha$-critical graph induces a rank inequality which is 
facet-defining for its stable set polytope.  Investigating a weighted
generalization, Lipt\'ak and Lov\'asz (2000, 2001) introduce critical
facet-graphs (which again produce facet-defining inequalities for
their stable set polytopes) and they establish a finite basis
theorem.  Second, Koppen (1995) describes a construction that
delivers from any $\alpha$-critical graph a facet-defining inequality
for the linear ordering polytope.  Doignon, Fiorini and Joret (2006)
handle the weighted case and thus define facet-defining graphs.  Here
we investigate relationships between the two weighted generalizations
of $\alpha$-critical graphs.  We show that facet-defining graphs (for
the linear ordering polytope) are obtainable from $1$-critical
facet-graphs (linked with stable set polytopes). We then use this
connection to derive various results on facet-defining graphs, the
most prominent one being derived from Lipt\'ak and Lov\'asz's finite
basis theorem for critical facet-graphs. At the end of the paper we
offer an alternative proof of Lov\'asz's finite basis theorem for
$\alpha$-critical graphs.
\end{abstract}

\section{Introduction}

A (finite, simple, undirected) graph $G$ is {\ac}  
if its stability number $\alpha(G)$ (defined as the maximum cardinality 
of a subset of mutually nonadjacent vertices) 
increases whenever an edge is removed from its edge set.
These graphs have several interesting structural properties, 
most of which being related to their {\em defect} $\delta = |G| - 2\alpha(G)$. 
An important result of Lov\'asz~\cite{L78} shows for instance that
for every fixed defect $\delta \geq 1$, there exists a finite collection of graphs
from which  every connected {\ac} graph with defect $\delta$ can be derived
using a certain edge subdivision operation. 

One of the interests of {\ac} graphs lies in their connection with
facets of some polytopes arising in combinatorial optimization: 
Chv\'atal~\cite{C75} and Koppen~\cite{K95} showed how to obtain
facets of respectively the stable set and linear ordering polytopes
from (connected) {\ac} graphs. 
This link was investigated further in the recent years and
led to the introduction of two generalizations of {\ac} graphs, one called
{\em \cfg s}~\cite{LL00,LL01,S90} and the other {\em \fdg s}~\cite{CDF04,DFJ06}.
Graphs in both families  are vertex-weighted, and give rise to facets
of the stable set and linear ordering polytopes, respectively. 

Although examples show that the classes of {\cfg s} and {\fdg s} are (inclusion-wise)
incomparable, some of the known results on their respective structures are intriguingly
similar (see e.g.~\cite{LL01} and~\cite{DFJ06}). 
The purpose of this paper is to explain precisely how {\cfg s} and {\fdg s} are related to each other.

In a recent contribution, Fiorini~\cite{F06-recycle} already showed that
a subclass of the former, which we call {\em $1$-\cfg s}, are {\fdg s}.
Here we prove a converse result: Every {\fdg} can be obtained from some {$1$-\cfg}
using a simple contraction operation.  
This connection conveys a great deal of information on {\fdg s}.
In particular, the main result of Lipt\'ak and Lov\'asz~\cite{LL00},
an extension of Lov\'asz's finite basis theorem to the class of {\cfg s},
translates naturally to {\fdg s}. 

The paper is organized as follows. We first give the necessary
definitions and preliminaries in Section~\ref{sec-prelim}.
We then present in Section~\ref{sec-connection} 
our main result which relates {\fdg s} to {$1$-\cfg s},
and use it to derive new results on {\fdg s} from the theory of {\cfg s}.
Finally, in Section~\ref{sec-basis}, we go back to {\ac} graphs
and offer an alternative proof for
the finite basis theorem of Lov\'asz. The latter theorem 
is not only at the heart of the theory of {\ac} graphs, but also
a key ingredient in Lipt\'ak and Lov\'asz's proof for
the extension of the result to {\cfg s}.

\section{The stable set and linear ordering polytopes}
\label{sec-prelim}

In this section, we define the stable set and linear 
ordering polytopes, the two classes of weighted 
graphs under consideration, and the corresponding
facets. We also state the main known results on these 
two classes of weighted graphs.

\subsection{The stable set polytope and \cfg{}s}
\label{sec-SS-cfg}

The {\em stable set polytope $\STAB(G)$} of a graph $G$ is defined as 
the convex hull of the incidence vectors of all stable sets of $G$.
In other words, letting $V$ and $E$ respectively denote the vertex 
and edge sets of $G$, the stable set polytope of $G$ is the integer
hull of the polytope 
$$
P := \{x \in \R^V \mid x_i + x_j \le 1 \ \forall ij \in E,\ 
0 \le x_i \le 1 \ \forall i \in V\},
$$
that is, 
$$
\STAB(G) = \mathrm{conv}(P \cap \Z^V).
$$

A central question in polyhedral combinatorics is to determine the
facets of $\STAB(G)$. While this is believed to be impossible in 
general for complexity theoretic reasons, see, e.g., Papadimitriou 
and Yannakakis~\cite{PY84}, there exist numerous 
published works focussing on special classes of graphs or special 
families of facets. A large number of these papers are concerned with 
facets defined by {\em rank inequalities\/}, that is, inequalities of 
the form
$$
\sum_{v \in S} x_v \le \alpha(G[S])
$$ 
for some $S \subseteq V$. In particular, one might ask when the 
rank inequality obtained for $S = V$, i.e., $\sum_{v \in V} x_v 
\le \alpha(G)$, defines a facet of $\STAB(G)$. In 1975, 
Chv\'atal~\cite{C75} showed that this is the case whenever 
$G$ is a connected {\ac} graph, where $G$ is said to 
be {\em \ac} if $\alpha(G - e) > \alpha(G)$ for every $e\in E(G)$. 
Thus {\ac} graphs are of particular relevance to the polyhedral 
theory of the stable set polytopes.
The literature on these graphs is quite rich, most contributions 
dating back to the 60's and 70's (see~\cite{LP86} for a survey).
Two concepts turn out to be of key importance for the study of
\ac{} graphs: an invariant called the `defect' and an operation 
known as taking `\os{}s'. The {\em defect} of a graph $G$ is 
defined as $\delta = |G| - 2\alpha(G)$. This invariant is 
always nonnegative when $G$ is \ac{}. An {\em \os} of a graph 
$G$ is any graph that can be obtained from $G$ by replacing 
edges with odd-length paths. Any {\os} of a connected {\ac} 
graph $G$ with at least three vertices is again {\ac} and has
the same defect (see, e.g.,~\cite{LP86}). A central result, due 
to Lov\'asz~\cite{L78}, shows essentially that {\ac} graphs 
are naturally classified by their defect. It is known as the 
{\em finite basis theorem for \ac{} graphs}. 

\begin{theorem}[Lov\'asz~\cite{L78}]
\label{th-acg-basis}
For every integer $\delta \ge 1$, there exists a finite 
collection of graphs such that every connected {\ac} graph 
with defect $\delta$ is an {\os} of a graph in the collection.
\end{theorem}

Let $G \preceq H$ whenever $H$ is an \os{} of $G$. This
defines a partial order on graphs. Consider the set 
of all connected {\ac} graphs partially ordered by $\preceq$. 
The graphs with fixed defect form a partition of this poset 
into upper monotone sets. Then the finite basis theorem amounts
to say that each of these upper monotone sets contains a finite 
number of minimal elements.

Let $G$ be any graph. Now consider a weight function $a$ on 
the vertices of $G$, that is, a function $a : V \to \Z_+$. 
The pair $(G,a)$ is referred to as a {\em (vertex)-weighted graph}.
From now on, in order to avoid some trivialities, we will always
assume that weighted graphs have at least three vertices and 
$a(v) > 0$ for all vertices $v$. Letting $\alpha(G,a)$ denote the 
maximum weight of a stable set in $G$, the weighted graph $(G,a)$ is 
said to be {\em critical} if $\alpha(G - e, a) > \alpha(G, a)$ for all 
edges $e$. Moreover, $(G,a)$ is said to be a {\em facet-graph} if 
the inequality
$$
\sum_{v \in V} a(v)\,x_{v} \le \alpha(G,a)
$$ 
defines a facet of $\STAB(G)$ and $G$ is connected (recall 
that we also assume that $G$ contains at least three vertices 
and the weights are positive). The \cfg{}s are the natural 
weighted counterpart of \ac{} graphs. Many results from the 
theory of \ac{} graphs were extended to \cfg{}s, see the works 
of Sewell~\cite{S90} and Lipt\'ak and Lov\'asz~\cite{LL00, LL01}. 

The {\em defect} of a weighted graph $(G,a)$ is defined as 
$\delta=a(V(G)) - 2 \alpha(G,a)$. As was the case for \ac{} 
graphs, this invariant turned out to be crucial for
studying \cfg{}s. The following result reveals much of the 
structural information conveyed by the defect of a {\cfg}.

\begin{theorem}[Lipt\'ak and Lov\'asz~\cite{LL01}]
\label{th-cfg-bound}
If $(G,a)$ is a {\cfg} with defect $\delta$, then $\deg(v) \le a(v) + \delta \le 2\delta$ for every $v\in V(G)$,
and $\deg(v)\le 2\delta - 1$ when $\delta > 1$.
\end{theorem}

Let $(G,a)$ be a weighted graph $(G,a)$ and $e$ be one
of its edges. The {\em strength} of the edge $e$ is defined as 
$\alpha(G-e,a)-\alpha(G,a)$. Notice that if $(G,a)$ is a 
\cfg{} then the strength of any of its edges is positive.
Consider now the following operation on $(G,a)$:
select some of its edges, and replace each with a path 
of length 3 where the two new vertices have weight equal 
to the strength of the edge. The resulting weighted graph is 
referred to as an {\em elementary \os} of $(G,a)$. 
We say that a weighted graph is an {\em \os} of $(G,a)$ if 
it is obtained from $(G,a)$ by applying the operation finitely 
many times.

\begin{lemma}[Wolsey~\cite{W76}]
\label{lem-cfg-os}
Every elementary \os{} of a \cfg{} is again a \cfg{} with 
the same defect. The three new edges have the same strength
as the edge they replace.
\end{lemma}

The following result generalizes Lov\'asz's finite basis 
theorem for \ac{} graphs (Theorem~\ref{th-acg-basis}).

\begin{theorem}[Lipt\'ak and Lov\'asz~\cite{LL00}]
\label{th-cfg-basis}
For every integer $\delta \ge 1$, there exists a finite collection of {\cfg s} 
such that every {\cfg} with defect $\delta$
is an {\os} of a graph in the collection.
\end{theorem}

Such a collection of graphs is (explicitly) known for 
$\delta=1,2$ only. Using Theorem~\ref{th-cfg-bound}, it is 
not difficult to check that {\cfg s} with defect 1 are the 
odd cycles with the all-one weighting, that is, the {\os s} 
of $(K_{3}, \one)$.  
For $\delta=2$, Sewell~\cite{S90} proved the following.

\begin{theorem}[Sewell~\cite{S90}]
\label{th-cfg-basis-def2}
Every {\cfg} with defect $2$ is an {\os} of one of the graphs 
depicted in Figure~\ref{fig-cfg-basis-def2}.
\end{theorem}

\begin{figure}
\centering
\includegraphics[width=0.6\textwidth]{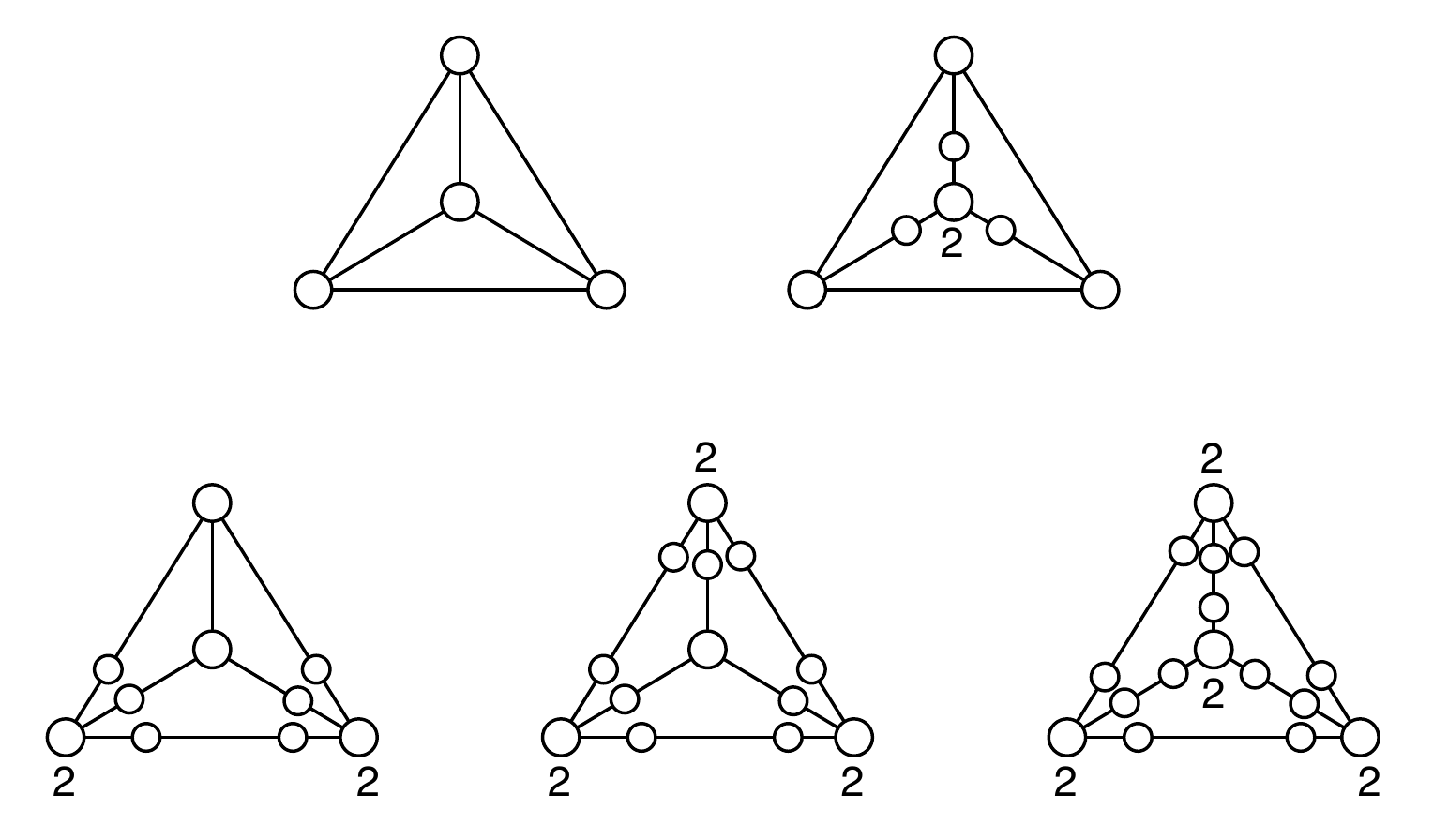}
\caption{\label{fig-cfg-basis-def2}A basis for {\cfg s} with defect 2 (only weights
different from 1 are indicated).}
\end{figure}

\subsection{The linear ordering polytope and \fdg{}s}

Given a complete directed graph with nonnegative weights on its arcs, 
the {\em linear ordering problem} asks to layout the vertices of the
graph on an oriented line in such a way that the total weight of the
arcs going from left to right is maximized. More precisely, solving 
the linear ordering problem consists in finding a strict linear ordering 
(that is, a spanning acyclic subtournament) of maximum total weight in
a given weighted complete directed graph. The 0/1-polytope 
naturally associated to this problem is known as the linear
ordering polytope. Let $N$ and $A$ respectively denote the node 
and arc set of the complete directed graph given as input and let 
$n = |N|$. Then the {\em linear ordering polytope} $\Plo^N$ (we sometimes
denote it simply by $\Plo^n$) is the integer hull of the polytope 
$$
Q = \{x \in \R^A \mid x_{ij} + x_{jk} + x_{ki} \le 2 \ \forall \{ij, jk, ki\}
\subseteq A,\ x_{ij} + x_{ji} = 1\ \forall ij \in A,\ x_{ij} \le 1\ 
\forall ij \in A\}.
$$
Equivalently, the linear ordering polytope is the convex hull of the
incidence vectors of all strict 
linear orderings contained in $D = (N,A)$. The literature dealing
with the polyhedral structure of the linear ordering polytope is 
quite abundant (with an approximate number of 50 references),
although not as abundant as the literature on the stable set 
polytope. A prominent class of facets for this former polytope 
are the so-called fence inequalities which where independently
discovered by Gr\"otschel, J\"unger and Reinelt~\cite{GJR85-lop}
and Cohen and Falmagne~\cite {CF90}. They were generalized 
in two different ways by Leung and Lee~\cite{LL94} (also Suck~\cite{S92})
and Koppen~\cite{K95}. Then the authors of the present paper
proposed a further generalization unifying the two generalizations 
mentioned above, following an idea of Christophe, Doignon and 
Fiorini~\cite{CDF04}. The resulting class of inequalities is 
known as the {\em graphical inequalities}. 
We give a definition of these inequalities in the next paragraph.
To avoid any confusion,
let us emphasize that, while arc-weighted directed graphs briefly appeared
in the definition of the linear ordering problem, all weighted graphs considered
in the sequel will be vertex-weighted undirected  graphs (as in Section~\ref{sec-SS-cfg}).

The {\em worth} of a subset $S$ of vertices of a weighted graph 
$(G,a)$ is defined as $a(S) - ||S||$, where $||S|| = |E(G[S])|$
denotes the number of edges of $G$ with both ends in the set $S$. 
The maximum worth of a set of vertices in $(G,a)$ is denoted by 
$\talpha(G,a)$. In other words, we let
$$
\talpha(G,a) := \max_{S \subseteq V(G)} \{ a(S) - ||S|| \}.
$$
Notice $\talpha(G,a) \ge \alpha(G,a)$ because $||S|| = 0$ 
whenever $S$ is a stable set. As precedingly, let $V$ and $E$ 
respectively denote the vertex and edge set of $G$. Suppose that 
$N$ contains $V$ and, furthermore, a set $V'$ disjoint from $V$ 
and of the same cardinality. Let $v \mapsto v'$ denote any 
bijection from $V$ onto $V'$. The {\em graphical inequality}
defined by $(G,a)$ then reads
\begin{equation}
\label{eq-graphical}
\sum_{v \in V} a(v)\,x_{vv'} - \sum_{vw \in E} (x_{vw'} + x_{wv'}) 
\le \talpha(G,a).
\end{equation}
A weighted graph $(G,a)$ is  a {\em \fdg} if the
corresponding graphical inequality defines a facet of the linear 
ordering polytope (as stated above, we also assume $|V(G)| \geq 3$ 
and $a(v) > 0$ for all $v\in V(G)$). Suppose for a moment that 
$a(v) = 1$ for all vertices $v$. That is, $a$ is the all-one 
function $\one$. Koppen \cite{K95} showed that in this case 
$(G,a)$ is facet-defining precisely when $G$ is a connected 
\ac{} graph distinct from $K_2$. This result is reminiscent of the 
aforementioned result of Chv\'atal \cite{C75} on the stable set polytope. 
This is not a coincidence, as we now explain. 

\begin{theorem}[Fiorini~\cite{F06-recycle}, Corollary~16]
\label{th-sam}
Let $(G,a)$ be a \cfg{} with $G=(V,E)$. As above, assume that 
$V$ is contained in $N$ and $v \mapsto v'$ is a bijection between
$V$ and a subset $V'$ of $N$ which is disjoint from $V$. Finally, 
for an edge $e \in E$, let $s(e)$ denote its strength. Then there 
exists a unique integer $\gamma$ such that the inequality
\begin{equation}
\label{eq-recycle}
\sum_{v \in V} a(v)\,x_{vv'} - \sum_{vw \in E} 
s(vw)\,(x_{vw'} + x_{wv'}) \le \gamma
\end{equation}
is facet-defining for the linear ordering polytope.
\end{theorem}

A \cfg{} $(G,a)$ is said to be {\em $k$-critical} if the
strength of any of its edges is at most $k$. Suppose that 
$(G,a)$ is a $1$-\cfg{} and consider inequality 
\eqref{eq-recycle}. Because the strength of every 
edge of $(G,a)$ equals $1$, the left-hand side of 
Eq.~\eqref{eq-recycle} equals the left-hand side of 
Eq.~\eqref{eq-graphical}, that is, the graphical 
inequality associated to $(G,a)$. It follows that 
$\gamma = \talpha(G,a)$ and thus Eq.~\eqref{eq-recycle} 
is a facet-defining graphical inequality and $(G,a)$ 
is a \fdg. This shows that 
$1$-\cfg{}s are always {\fdg s}.
In the next section we prove that, conversely, 
any \fdg{} $(G,a)$ has a \uos{} which is a $1$-\cfg{}. 

\section{The connection and some of its consequences}
\label{sec-connection} 

In this section we state and prove our main result
which relates \fdg{}s to $1$-\cfg{s}. We then derive
new results on \fdg{}s from Theorems \ref{th-cfg-bound}
and \ref{th-cfg-basis}. Thus, in particular, we derive
a finite basis theorem for \fdg{}s. At the end of the 
section, we provide the basis for subdefects $1$ and $2$.

Let $(G,a)$ be an arbitrary weighted graph. The {\em subdefect} 
$(G,a)$ is defined as $\lambda = a(V(G)) - 2\talpha(G,a)$. Notice 
that the subdefect of a weighted graph never exceeds its defect
(hence the name).
A {\em \uos} of $(G,a)$ is any graph obtained from $(G,a)$ 
by replacing edges with odd-length paths, where the new vertices 
have weight 1. 
Conversely, a graph $(G',a')$ is said to be a {\em \invos} 
of $(G,a)$ if $(G,a)$ is a {\uos} of $(G',a')$.

The following properties of {\fdg s} were proved in~\cite{CDF04} 
(see also~\cite{DFJ06}).

\begin{lemma}[Christophe, Doignon and Fiorini \cite{CDF04}]
\label{lem-prop}
Let $(G,a)$ be a {\fdg}. Then
\begin{enumerate}[(A)]
\item \label{system}the only solution to the system 
$$
\left\{ \sum_{v\in T} y_{v} +  \sum_{e\in E(T)} y_{e} = \talpha(G,a) \mid 
T\subseteq V(G),\ T \textrm{ maximum worth set}  \right\}
$$
is the trivial solution: $y_{v} = a(v)$ for all $v\in V(G)$, $y_{e} = -1$ for all $e\in E(G)$;
\item \label{TScontains} 
for every $uv\in E(G)$ and $X \subseteq \{u,v\}$, there exists a 
maximum worth set $T\subseteq V(G)$ with $T \cap \{u,v\} = X$;
\item \label{deg} 
$\deg(v) \geq 2$ for every $v\in V(G)$; 
\item \label{OS} any {\uos} of $(G,a)$ is also facet-defining with the same subdefect, and 
\item \label{invOS} any {\invos} of $(G,a)$ is also facet-defining with the same 
subdefect.
\end{enumerate}
\end{lemma}

By Lemma \ref{lem-cfg-os}, the notions of \os{} and \uos{} 
coincide for $1$-\cfg{}s. The next lemma shows that this
is also the case for the defect and subdefect.

\begin{lemma}
\label{lem-alpha-beta}
If $(G,a)$ is a $1$-\cfg{} then its defect equals its subdefect, 
and hence $\alpha(G,a)=\talpha(G,a)$.
\end{lemma}
\begin{proof}
Consider any (unit) {\os} $(G',a')$ of $(G,a)$ where no edge 
of $G$ remains. Because every edge in $(G',a')$ is incident 
to at least one new vertex, which all have weight $1$, we
have $\alpha(G',a')=\talpha(G',a')$. Indeed, any set 
$S \subseteq V(G')$ can be turned in a stable set whose
worth is at least that of $S$ by iteratively removing
any vertex of weight 1 adjacent to some other vertex in $S$. 
By Lemma~\ref{lem-cfg-os}, $(G',a')$ is a {\cfg} with 
the same defect as $(G,a)$. Now Theorem~\ref{th-sam} 
implies that $(G',a')$ is a {\fdg}. Then, by Lemma~\ref{lem-prop}\eqref{invOS}, 
$(G,a)$ is also a \fdg{} and has the same 
subdefect as $(G',a')$. Since the defect of $(G',a')$ 
equals its subdefect, we deduce that the same holds 
for $(G,a)$. The lemma follows.
\end{proof}

We now turn to the main contribution of this paper: a
precise connection between \fdg{}s and \cfg{}s.

\begin{prop}
\label{prop-corresp}
A weighted graph is facet-defining if and only if it is a {\invos} of a $1$-\cfg{}. Moreover, the subdefect of the former equals the defect of the latter.
\end{prop}
We remark that there are {\fdg s} which are not facet-graphs, for instance
the last two graphs in Figure~\ref{fig-fdg-basis-2}.
\begin{proof}[Proof of Proposition~\ref{prop-corresp}]
Assume first that a graph $(G',a')$ is a {\invos} of a $1$-critical facet-graph 
$(G, a)$. Then $(G,a)$ is a {\fdg} (Theorem~\ref{th-sam}), and so
is  $(G',a')$ (Lemma~\ref{lem-prop}\eqref{invOS}).
Moreover, $(G, a)$ has equal subdefect and defect (Lemma~\ref{lem-alpha-beta}).
Also, $(G',a')$ and $(G,a)$ have same subdefect (Lemma~\ref{lem-prop}\eqref{invOS}).
Hence, the subdefect of $(G',a')$ equals the defect of $(G,a)$.

Assume now that $(G',a')$ is a {\fdg} and let $(G,a)$ be 
the \uos{} of $(G',a')$ obtained by replacing each edge 
with a path of length 3 and giving a weight of 1 to the 
new vertices. Following Lemma~\ref{lem-prop}\eqref{OS}, 
$(G, a)$ is also facet-defining. Moreover, as in the proof of
Lemma~\ref{lem-alpha-beta} (the roles of $(G,a)$ and 
$(G',a')$ are now interchanged), we have $\talpha(G,a) 
= \alpha(G,a)$. Observe, in passing, that the same holds for 
all spanning subgraphs of $(G,a)$. Now consider an edge $e$ 
of $(G,a)$. Then, by Lemma~\ref{lem-prop}\eqref{TScontains}, 
we have $\beta(G-e,a) = \beta(G,a)+1$. On the other hand,
we also have $\alpha(G-e,a) = \beta(G-e,a)$ as $G-e$ is a 
spanning subgraph of $(G,a)$.
Hence the strength of every edge of $(G,a)$ equals $1$. 
We now show that $(G,a)$ is also a facet-graph.

Arguing by contradiction, assume that $(G,a)$ is not a facet-graph. 
This means that $(G,a)$ does not contain $|G|$ linearly independent
maximum weight stable sets (since the stable set polytope is full-dimensional). It follows then that the system 
$$
\left\{ \sum_{v\in S} y_{v} = \alpha(G,a) \mid S\subseteq V(G),
\ S \textrm{ maximum weight stable set} \right\}
$$
has a solution $\tilde y$ distinct from the solution $y_{v}=a(v)$ for all $v\in V(G)$. 

For each edge $e$ of $G$, pick a vertex $t_{e}$ of weight $1$ 
incident to $e$. Extend now $\tilde y$ to a vector in 
$\mathbb{R}^{V(G) \cup E(G)}$ by letting $\tilde y_{e} = 
-\tilde y_{t_{e}}$ for every edge $e$. Consider any maximum 
worth set $T$ of $(G, a)$ and let $S := T \setminus \{t_{e_{1}}, 
t_{e_{2}}, \dots, t_{e_{k}}\}$, where $E(T) = \{e_{1}, e_{2}, 
\dots, e_{k}\}$. Since $S$ is a maximum weight stable set, we 
obtain
$$
\sum_{v\in T} \tilde y_{v}  + \sum_{e\in E(T)} \tilde y_{e}
= \sum_{v\in T} \tilde y_{v}  - \sum_{e\in E(T)} \tilde y_{t_{e}} 
= \sum_{v\in S} \tilde y_{v} = \alpha(G,a) = \talpha(G,a).
$$
Hence, this extended vector $\tilde y$ is a non trivial solution 
of the system defined in Lemma~\ref{lem-prop}\eqref{system}, 
contradicting the fact that $(G,a)$ is facet-defining. Therefore, 
$(G,a)$ is a $1$-\cfg{}. This concludes the proof.
\end{proof}

Several structural properties of {\fdg s} derive from 
Proposition~\ref{prop-corresp} combined with known results 
on {\cfg s}, as we know illustrate. We first note a direct 
corollary of Theorem~\ref{th-cfg-bound}:

\begin{cor}
\label{cor-fdg-bound}
If $(G,a)$ is a {\fdg} with subdefect $\lambda$, 
then $\deg(v) \le a(v) + \lambda \le 2\lambda$ for every $v\in V(G)$,
and $\deg(v)\le 2\lambda - 1$ when $\lambda > 1$.
\end{cor}

One of the main interests of Proposition~\ref{prop-corresp} is that
the finite basis theorem for {\cfg s} (Theorem~\ref{th-cfg-basis}) 
extends naturally to {\fdg s}.

\begin{cor}
\label{cor-fdg-basis}
For every integer $\lambda\geq 1$, there exists a finite collection of 
facet-defining graphs such that every facet-defining graph $(G,a)$ with 
subdefect $\lambda$ is a unit odd subdivision of a graph in the collection.
\end{cor}

Before turning to the proof of Corollary~\ref{cor-fdg-basis}, 
we need the following result:

\begin{lemma}
\label{lem-cutset}
In a {\fdg} a cutset cannot induce $K_{2}$. 
\end{lemma}
\begin{proof}
Let $(G,a)$ be a {\fdg}. 
Arguing by contradiction, assume that $G=G_1 \cup G_2$ with 
$V(G_1)\cap V(G_2)=\{v,w\}$ and $vw\in E(G)$. Let $\talpha 
:= \talpha(G,a) $ and $V_i:=V(G_i)$, $E_i:=E(G_i)$ for $i=1,2$.

The {\mws s} of $(G,a)$ can be classified in 4 categories, 
according to their intersection with $\{v,w\}$ (which can be 
$\varnothing, \{v\}, \{w\}$ or $\{v,w\}$). It follows from 
Lemma~\ref{lem-prop}\eqref{TScontains} that $(G,a)$ has at 
least one {\mws} in each category. For $X \subseteq \{u,v\}$ 
and $i=1,2$, we define $c_{X}^{i}$ as
$$
c_{X}^{i}:= \big(a(T \cap V_{i})  - ||T \cap V_{i}||\big) - \big(a(X) - ||X||\big), 
$$
where $T$ is any {\mws} of $(G,a)$ with $T \cap \{u,w\} = X$. Notice that,
since $\{v,w\}$ is a cutset of $G$, the value of $c_{X}^{i}$ is independent of the particular
choice of $T$.

Pick any $\gamma^{1} \in \R$ distinct from $1$ and let, 
using the fact that $c_{\varnothing}^{2}\neq 0$,  
$$
\gamma^{2} := \frac{ \talpha - \gamma^{1}c_{\varnothing}^{1} }{ c_{\varnothing}^{2} }.
$$
Define a vector $y\in \R^{V(G) \cup E(G)}$ as follows:
\begin{align*}
y_u &:=  \gamma^{i} \cdot a(u)  & \textrm{ for } i=1,2 \textrm{ and } u\in V_i \setminus \{v,w\}; \\
y_e &:= \gamma^{i} \cdot (-1)  & \textrm{ for } i=1,2 \textrm{ and } e\in E_i \setminus \{vw\}; \\
y_u &:= \talpha  - \gamma^{1}  c_{\{u\}}^{1} - \gamma^{2}  c_{\{u\}}^{2} 
& \textrm{ for }  u \in \{v,w\}; \\
y_{vw} &:= \talpha  - \gamma^{1}  c_{\{v,w\}}^{1} -  \gamma^{2}  c_{\{v,w\}}^{2} - y_v - y_w. 
\end{align*}
This vector $y$ is a non trivial solution to the system of Lemma~\ref{lem-prop}\eqref{system}, a contradiction.
\end{proof}

\begin{proof}[Proof of Corollary~\ref{cor-fdg-basis}]
In virtue of Lemma~\ref{lem-prop}\eqref{deg}, every vertex of 
a {\fdg} has degree at least 2. This is in particular true for 
$1$-\cfg{}s. Now consider some (sub)defect $\lambda \geq 1$. 
By Theorem~\ref{th-cfg-basis}, the number of vertices with 
degree at least 3 in a $1$-\cfg{} with defect $\lambda$ is 
bounded from above by some constant $c_{\lambda}$ that depends 
only on $\lambda$. 

We call an edge {\em remote} if both of its ends have
degree 2. Denote by $\mathcal{B}_{\lambda}$ the set of 
\fdg{}s with subdefect $\lambda$ having no remote edge. 
Every facet-defining graph $(G,a)$ with subdefect 
$\lambda$ is a unit odd subdivision of some graph in 
$\mathcal{B}_{\lambda}$, as easily proved by induction 
on $|G|$: either $(G,a) \in \mathcal{B}_{\lambda}$ or 
$(G,a)$ has a remote edge $uv$. In the latter case, we 
find an induced path $u'uvv'$ in $G$, as otherwise there 
would be a cutset inducing $K_{2}$, which Lemma~\ref{lem-cutset}
forbids. Now, by `shrinking' this path (i.e. removing $u,v$ 
and adding the edge $u'v'$) and using Lemma~\ref{lem-prop}%
\eqref{invOS}, we are done by induction. Hence, 
$\mathcal{B}_{\lambda}$ is a basis for {\fdg s} with 
subdefect $\lambda$. 
 
We know from Proposition~\ref{prop-corresp} that any graph 
$(G,a) \in \mathcal{B}_{\lambda}$ is a shrinking of a 
$1$-critical facet-graph, and thus that the number of 
vertices with degree at least 3 in $(G,a)$ is bounded by 
$c_{\lambda}$. Since $(G,a)$ has no remote edge, we 
deduce $|G| \le c_{\lambda} + 2\lambda {c_{\lambda} \choose 2}$
(cf. Corollary~\ref{cor-fdg-bound}), 
and that $\mathcal{B}_{\lambda}$ is finite.
\end{proof}

Similarly as for {\cfg s}, Corollary~\ref{cor-fdg-bound} implies that
{\fdg s} with subdefect 1 are the odd cycles with unit weights.
We note that Theorem~\ref{th-cfg-basis-def2} shows in particular that every
{\cfg} with defect 2 is $1$-critical. Hence, we obtain the following corollary.

\begin{cor}
\label{cor-fdg-basis-2}
Every {\fdg} with subdefect $2$ is a {\uos} of a graph depicted in Figure~\ref{fig-fdg-basis-2}.
\end{cor}

\begin{figure}
\centering
\includegraphics[width=0.6\textwidth]{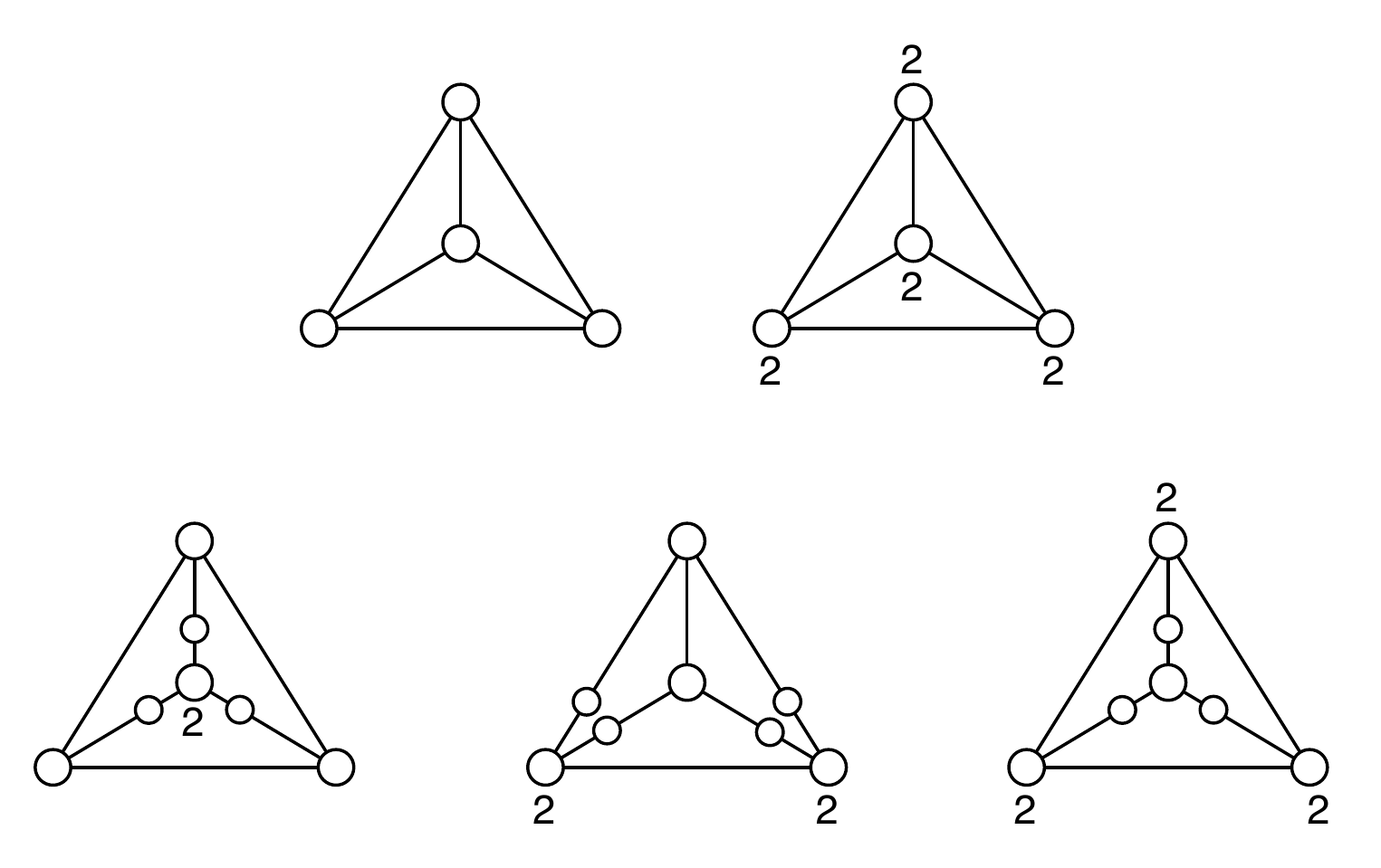}
\caption{\label{fig-fdg-basis-2}A basis for {\fdg s} with subdefect 2 (only weights
different from 1 are indicated).}
\end{figure}

\section{Finite basis for $\alpha$-critical graphs}
\label{sec-basis} 

As we have seen, the finite basis result for {\fdg s} (Corollary~\ref{cor-fdg-basis})
is a consequence of the corresponding theorem for {\cfg s}, Theorem~\ref{th-cfg-basis},
which was proved by Lipt\'ak and Lov\'asz~\cite{LL00}. The main step of their proof 
is a lemma which says roughly that every {\cfg} is the image of an {\ac} graph with 
the same defect under a particular well-behaved homomorphism. The result is then derived 
from Lov\'asz's finite basis theorem for $\alpha$-critical graphs (Theorem~\ref{th-acg-basis}).
Hence, the latter theorem is not only important  for {\ac} graphs, it is also a key result
for {\cfg s} and {\fdg s}. The purpose of this section is to present an alternative proof
for this theorem, restated as follows:

\begin{theorem}[Lov\'asz~\cite{L78}]
\label{th-L78-deg3}
For every $\delta \ge 1$, there exists a constant $c_{\delta}$ such that
every connected {\ac} graph with defect $\delta$ has at most $c_{\delta}$
vertices with degree at least 3.
\end{theorem}

This version implies the one given in Theorem~\ref{th-acg-basis}. Indeed,
every connected {\ac} graph which is minimal for the partial order $\preceq$ 
associated to the {\os} operation does not have two adjacent vertices with 
degree 2. Moreover, it is easily seen that, in a graph $G$, if two vertices $v$ and $w$
are not adjacent and have exactly the same neighbors, then any edge $e$ incident to $v$ or
$w$ is such that $\alpha(G-e) = \alpha(G)$. Hence, there are at most $r \choose 2$
vertices with degree 2 in a minimal connected {\ac} graph, where $r$ is
the number of vertices with degree at least 3.

The outline of our proof is as follows. We first relate the defect of an {\ac} 
graph $G$ to the maximum order of an acyclic tournament in a collection of 
directed graphs associated to $G$. We then use this relationship to transform 
the problem into a Ramsey-type problem on digraphs, which in turn follows from 
standard results in Ramsey theory. Let us emphasize that, while this gives a 
shorter and perhaps simpler proof of the existence of $c_{\delta}$, the value 
for $c_{\delta}$ that is implied by our proof is much larger than the one 
proved in~\cite{L78}.

In this section, by a {\mws} of a graph $G$ we mean a {\mws} of $(G,\one)$. 
A main ingredient in our proof of Theorem~\ref{th-L78-deg3} is the following 
simple lemma on sequences of {\mws s}. Interestingly, this lemma was originally 
introduced in a more general form in~\cite[Lemma~16]{DFJ06}, as a tool to study 
the subdefect of {\fdg s}.

\begin{lemma}[Doignon et al.~\cite{DFJ06}]
\label{lem-bound-defect}
Let $G$ be an {\ac} graph with defect $\delta$ and $T_{1},\dots,T_{k}$ be a 
sequence of {\mws s} (repetitions are allowed) such that for every vertex $u$ 
of $G$ there exist indices $i,j \in \{1,\dots, k\}$ with $u \in T_{i}$ and 
$u \notin T_{j}$. Then 
$$
\delta \geq \sum_{i=1}^{k} ||T_{i}|| - \sum_{j=3}^{k} ||X_{j}||,
$$
where $X_{j} :=  
\big( (T_{1} \cup \dots \cup T_{j-1}) \cap T_{j} \big) \cup
\big( (T_{1} \cap \dots \cap T_{j-1}) \setminus T_{j} \big)$.
\end{lemma}

\begin{proof}[Proof of Theorem~\ref{th-L78-deg3}]
Let $G$ be a connected {\ac} graph with defect $\delta$.
We want to show that the number of vertices with degree at least 3 is
bounded from above by some constant $c_{\delta}$ depending only on $\delta$.
To this aim, we may assume without loss of generality that $G$ has maximum degree exactly 3.
Indeed, nothing has to be proved if $G$ has no vertex with degree at least 3, and
if $v\in V(G)$ has degree at least 4, then we can simultaneously decrease the
number of vertices of degree more than 3 and increase the number of
vertices with degree at least 3 by {\em splitting} $v$: partition the neighbors of $v$
into two sets $N_{1},N_{2}$, each of cardinality at least 2, remove $v$, add three new vertices
$v_{1},v_{2},v'$, and link $v_{i}$ to $v'$ and the vertices of $N_{i}$, for $i=1,2$.
It is easily seen that this operation keeps a graph {\ac} and does not change the defect
(see, e.g., \cite{LP86} for a proof).

Denote by $v_{1}, \dots, v_{p}$ the vertices of $G$ with degree 3. We assume
that no two of them are adjacent, this can always be achieved by taking an appropriate
{\os} of $G$. Denote also by $e_{i,1}, e_{i,2}, e_{i,3}$ the three edges incident
to $v_{i}$, and let $T_{i,j}$ denote any maximum stable set of $G - e_{i,j}$.
Notice that $T_{i,j}$ is a {\mws} of $G$ with $E(T_{i,j}) = \{e_{i,j}\}$.

We define a digraph $D_{G}$ based on $G$ and the $T_{i,j}$'s. Its vertex 
set is the set of edges of $G$ which are incident to some degree-3 vertex, i.e.,
$$
V(D_{G}) = \{e_{i,j} \mid 1 \leq i\leq p \textrm{ and } 1\leq j\leq 3\},
$$
and for every distinct $i,k \in \{1,\dots,p\}$ and $j\in\{1,2,3\}$, we put an arc from $e_{i,j}$ 
to $e_{k,\ell}$ for all $\ell \in \{1,2,3\}$
whenever
$$
\textrm{either} \quad (v_{k} \in T_{i,j+1} \textrm{ and } v_k \in T_{i,j+2})
\quad \textrm{or} \quad (v_{k} \notin T_{i,j+1} \textrm{ and } v_k \notin T_{i,j+2}),
$$
where indices are taken mod 3. Moreover, we color the arc $(e_{i,j},e_{k,\ell})$ 
red in the first case, blue in the second.

An {\at} $J$ in $D_{G}$ is {\em \adm} if $J$ contains at most one of the three 
vertices $e_{i,1}, e_{i,2}, e_{i,3}$, for $1\leq i\leq p$. In addition, $J$ is 
said to be {\em red} (resp.\ {\em blue}) if all its arcs are colored red 
(resp.\ blue). Our main tool is the following observation:

\begin{claim}
\label{claim-connection}
If $J$ is a red or blue {\adm} {\at} in $D_{G}$, then $ |J| \leq \delta$.
\end{claim}
\begin{proof}
By renaming the indices if necessary, we may assume $V(J) = \{e_{i,1} \mid
1 \le i \le t\}$ and $A(J) = \{(e_{i,1},e_{k,1}) \mid 1 \le i < k \le t\}$.
Let $\{w_{1},\dots, w_{\ell}\} := V(G) \setminus \{v_{1}, \dots, v_{t}\}$
and, for $1 \le i \le \ell$, let $S_i$ (resp.\ $S^{'}_{i}$) be any maximum 
stable set of $G$ with $w_{i} \in S_{i}$ (resp.\ $w_{i} \notin S^{'}_{i}$). 
Now consider the following sequence of {\mws s} of $G$:
$$
T_{1,2}, T_{1,3} - v_{1}, T_{2,2}, T_{2,3} - v_{2}, \dots, T_{t,2}, T_{t,3} - v_{t},
S_{1}, S^{'}_{1}, \dots, S_{\ell}, S^{'}_{\ell}.
$$
For the sake of clarity, we will commit a slight abuse of notation and denote 
by $T_{i}$ the $i$-th set of the above sequence of $k := 2t + 2\ell$ sets. 
Notice that, by construction, our sequence of maximum worth sets satisfies 
the assumption of Lemma~\ref{lem-bound-defect}. Also, if $uv \in E(T_{i})$ 
then $u,v \notin T_{i+1}$, for $1 \leq i < k$. Defining $X_{j}$ as in 
Lemma~\ref{lem-bound-defect}, this implies that, for $j \geq 3$, if we 
have $uv \in E(X_{j})$, then we also have $uv \in E(T_{j})$. Hence, 
$E(X_{j}) \subseteq E(T_{j})$.

Using Lemma~\ref{lem-bound-defect}, we obtain: 
$$
\delta \geq ||T_{1}|| + ||T_{2}|| + \sum_{j=3}^{k} (||T_{j}|| - ||X_{j}||)
= 1+  \sum_{j=3}^{k} |E(T_{j}) \setminus E(X_{j})|. 
$$
Each term in the last sum is nonnegative. We now prove that at
least $t - 1$ of them are positive, which clearly implies 
the claim. Pick some $i \in \{2, \dots, t\}$ and denote by $x$ 
the end of the edge $e_{i,2}$ that is distinct from $v_{i}$.
If $J$ is red, then by the definition of $D_{G}$ we have $v_{i} 
\in T_{j}$ for $1 \leq j \leq 2i - 2$. Since $T_{2i-1}$ (which 
equals $T_{i,2}$) is the only set in our sequence of {\mws s} that 
contains both ends of $e_{i,2}$, we deduce $x \notin T_{j}$ for 
$1 \leq j \leq 2i - 2$. This shows $x \notin X_{2i - 1}$, and hence 
$e_{i,2} \in E(T_{2i-1}) \setminus E(X_{2i-1})$. Similarly, 
if $J$ is blue then it follows from the definition of $D_{G}$ that
$v_{i} \notin T_{j}$ for $1 \leq j \leq 2i - 2$, which implies 
$v_{i} \notin X_{2i - 1}$, and again $e_{i,2} \in E(T_{2i-1}) 
\setminus E(X_{2i-1})$. 
\end{proof}

By the above claim, to prove Theorem~\ref{th-L78-deg3} it is 
sufficient to show that if $G$ has many degree-3 vertices, then
there is a large monochromatic {\adm} {\at} in $D_{G}$. As $D_{G}$ 
is ``almost'' a complete digraph, this sounds like a Ramsey-type 
property, and indeed we will reduce it to Ramsey's theorem. To 
this aim, we proceed with three claims. The first one is an easy 
consequence of Ramsey's theorem (see e.g.~\cite{HH74} for a proof).

\begin{claim}
\label{claim-DbarD}
For every $k \ge 1$, there exists a constant $d_{k}$ such that for every digraph $D$
with at least $d_{k}$ vertices, $D$ or its complement $\bar D$ contains an {\at} of order $k$.
\end{claim}

We say that $D'$ is a {\em blow-up} of a digraph $D$ if it can be obtained as follows: 
first create three vertices $v_{1},v_{2},v_{3}$ per vertex $v$ of $D$, then for each 
arc $(v,w)\in A(D)$,
choose some subset $I_{(v,w)}\subseteq \{1,2,3\}, I_{(v,w)} \neq \varnothing$, and add the arcs 
$(v_{i},w_{1}), (v_{i},w_{2}), (v_{i},w_{3})$ for every $i\in I_{(v,w)}$. Similarly as before,
we say that an acyclic tournament in $D'$ is {\em $D$-{\adm}} if for every vertex $v$ of $D$,
it contains at most one of the three corresponding vertices in $D'$.

Let us give some intuition on acyclic tournaments in blow-ups of digraphs.
If $I_{(v,w)} = \{1,2,3\}$ for each arc $(v,w)\in A(D)$ in the definition of the blow-up operation,
then $D'$ is simply the lexicographic product $D \otimes \bar K_{3}$ of $D$ with the
complement of $K_{3}$. In particular, in this case a $D$-{\adm} 
acyclic tournament of order $k$ in $D'$ is readily obtained from an 
acyclic tournament of order $k$ in $D$. The same holds more generally if,
for every $v \in V(D)$, we have
$\cap_{(v,w) \in A(D)} I_{(v,w)} \neq \varnothing$, because
then $D'$ has a subgraph isomorphic to $D$ which contains exactly one the
three vertices $v_{1},v_{2},v_{3}$ for each $v \in V(D)$. 
It turns out that this observation can essentially be extended to the case where
the sets $I_{(v,w)}$ are arbitrary nonempty subsets of  $\{1,2,3\}$: the digraph $D'$
will contain a $D$-{\adm} acyclic tournament of order $k$, provided $D$ contains
a large enough acyclic tournament.
This is a consequence of the following claim.

\begin{claim}
\label{claim-blow-up-at}
For every $k \ge 1$, there exists a constant $a_{k}$ such that for every acyclic tournament $D$
on at least $a_{k}$ vertices, any blow-up $D'$ of $D$ contains a $D$-{\adm} 
acyclic tournament of order $k$.
\end{claim}

\begin{proof}
We prove the claim by induction on $k$, the case $k=1$ being trivial.
For the inductive step, set $a_{k}:=3a_{k-1}+1$. 
Let $v \in V(D)$ be the unique vertex of $D$ with out-degree $|D|-1$, 
and let $v_{1},v_{2},v_{3}$ be the corresponding three vertices in $D'$. 
There is at least one of the latter three vertices,
say $v_{1}$, for which the set $S\subseteq V(D)$ of vertices of $D$ which correspond to 
the out-neighbors of $v_{1}$ in $D'$ has cardinality  at least $(a_{k} - 1)/3=a_{k-1}$. 
Let also $S'\subseteq V(D')$ be the set of out-neighbors of $v_{1}$ in $D'$.
The digraph $D'[S']$ is clearly a blow-up of $D[S]$, and by the induction hypothesis, 
$D'[S']$ contains a $D[S]$-{\adm} acyclic tournament of order $k-1$. 
Using $v_{1}$ and the latter subgraph we obtain a $D$-{\adm} acyclic tournament of order $k$ in $D'$.
\end{proof}

\begin{claim}
\label{claim-blow-up}
For every $k \ge 1$, there exists a constant $b_{k}$ such that for every digraph $D$
on at least $b_{k}$ vertices, all blow-ups of 
either $D$ or $\bar D$  contain a $D$-{\adm} acyclic tournament of order $k$.
\end{claim}
\begin{proof}
We claim that $b_{k} := d_{a_{k}}$ will do. Indeed, by Claim~\ref{claim-DbarD},
$D$ or $\bar D$ contains then an acyclic tournament 
on a set $T$ of $a_{k}$ vertices, say without loss of generality
$D$. Then, following Claim~\ref{claim-blow-up-at}, every blow-up of $D[T]$
contains a $D[T]$-{\adm} acyclic tournament of order $k$, and the same clearly holds if we replace 
$D[T]$ with  $D$.
\end{proof}

We now have everything we need to conclude. 
Let $R$ (resp. $B$) be the digraph on vertex set $\{v_{1}, \dots, v_{p}\}$
where there is an arc from $v_{i}$ to $v_{k}$ ($i\neq k$) if 
$v_{k}$ is in at least two (resp. at most one) of the three sets 
$T_{i,1}, T_{i,2}, T_{i,3}$. By the definition of $D_{G}$, the red 
and blue parts of $D_{G}$ are blow-ups of respectively $R$ and $B$. 

Since $R = \bar B$, if $p \geq b_{\delta+1}$ holds, then using 
Claim~\ref{claim-blow-up} we deduce that there exists a monochromatic 
{\adm} acyclic tournament of order $\delta + 1$ in $D_{G}$. But then, $G$
has defect at least $\delta +1$ by Claim~\ref{claim-connection}, a 
contradiction. Hence, $p < b_{\delta+1}$, and $c_{\delta} := b_{\delta + 1} - 1$ 
will do in the statement of Theorem~\ref{th-L78-deg3}.
\end{proof}

\bibliographystyle{plain}
\bibliography{fdg-cfg}

\end{document}